\documentclass[11pt]{amsart}

\usepackage[title,titletoc,toc]{appendix}
\usepackage[T1]{fontenc}
\usepackage[applemac]{inputenc}
\usepackage{amsmath}
\usepackage{amssymb}
\usepackage{amsfonts}
\usepackage[francais,english]{babel}
\usepackage{amsthm}
\usepackage{enumerate}
\usepackage{accents}
\usepackage{stackengine}
\usepackage{tikz} 
\usepackage{pdfsync}
\usepackage{graphicx,color}

\usepackage[all]{xy}

\usepackage[top=3cm, bottom=2cm, left=3cm, right=3cm]{geometry}

\newtheorem{theorem}{Theorem}

\newtheorem{proposition}[theorem]{Proposition}
\newtheorem{lemma}[theorem]{Lemma}
\newtheorem{corollary}[theorem]{Corollary}
\newtheorem{remark}[theorem]{Remark}

\newcommand{\RR}{\mathbb{R}}
\newcommand{\R}{\mathbb{R}}
\newcommand{\N}{\mathbb{N}}
\newcommand{\Z}{\mathbb{Z}}

\newcommand{\eps}{\varepsilon}

\newcommand{\supp}{{\rm supp\ }}

\newcommand{\argmin}{{\rm argmin\ }}

\newcommand{\resc}{{\bf resc}}

\newcommand{\Em}{{\mathcal E}_m}

\newcommand{\mF}{{\mathcal F}}

\newcommand{\mFr}{{\mathcal F_{\resc}}}

\newcommand{\ird}{\int_{\mathcal{\R^d}}}
\newcommand{\prob}{\mathcal{P}(\R^d)}

\newcommand{\ds}{\displaystyle}

\def\ird{\int_{\R^d}}

\title{Free Energies and the Reversed HLS Inequality}

\author{J. A. Carrillo}
\address{Department of Mathematics, Imperial College London, London SW7 2AZ, UK}
\email{carrillo@imperial.ac.uk}
\author{M. G. Delgadino}
\address{Department of Mathematics, Imperial College London, London SW7 2AZ, UK}
\email{m.delgadino@imperial.ac.uk}

\date{\today}
\thanks{JAC and MGD were partially supported by EPSRC grant number EP/P031587/1. The authors are very grateful to the Mittag-Leffler Institute for providing a fruitful working environment during the special semester \emph{Interactions between Partial Differential Equations \& Functional Inequalities}.}

\begin{document}
\begin{abstract}
We prove reversed Hardy-Littlewood-Sobolev inequalities by carefully studying the natural associated free energies with direct methods of calculus of variations. Tightness is obtained by a dyadic argument, which quantifies the relative strength of the entropy functional versus the interaction energy. The existence of optimizers is shown in the class of $\prob$. With respect to their regularity, we study conditions for optimizers to be bounded functions. In a related model, we show the condensation phenomena, which suggests that optimizers are not in general regular.
\end{abstract}
\maketitle
\section{Introduction}

Given parameters $m>0$ and $k>0$, in this paper we are concerned with the minimization problem for the family of free energies $\mF_{k,m}:\prob\to[-\infty,+\infty]$
\begin{equation}\label{freeen}
\mF_{m,k}[\rho]=\ird \frac{\rho^m(x)}{m-1}\;dx + \iint_{\RR^d\times\RR^d} \frac{|x-y|^k}{2k}\rho(x)\rho(y)\;dxdy\,.
\end{equation}
They are the natural Liapunov functionals associated to aggregation-diffusion equations with homogeneous entropies and interaction potentials. For the range $k\in (-d,0)$, these equations have been studied in detail in  \cite{CCH1,CCHCetraro,CHMV}. The case $k>0$ was briefly analyzed in \cite{CCH1}, where the \textit{fair-competition} ($m=1-\frac{k}{d}$) and \textit{diffusion dominated ranges} ($m< 1-\frac{k}d$) were considered. It is shown that in these regimes, the equations do not exhibit a critical mass phenomena as for the cases with $k<0$ and the classical Keller-Segel case corresponding to the logarithmic kernel $k=0$, see \cite{BCC08}. The results of this paper concerns the \textit{aggregation dominated regime} $m> 1-\frac{k}d$. 

The use of free energies to understand the long-time asymptotics of gradient flow equations of the form 
\begin{equation*}
	\partial_t \rho = \nabla\cdot\left( \rho (\nabla W*\rho)\right) + \nabla\cdot\left(\rho\nabla U'(\rho)\right), \qquad t>0\,,
\end{equation*}
where $W\colon\R^d\to (-\infty,\infty]$ is the interaction potential and $U\colon[0,\infty)\to \R$ is the entropy functional, has attracted lots of attention in the last 20 years \cite{AGS,CaMcCVi03,CaMcCVi06,Villani03}. The connection to Hardy-Littlewood-Sobolev type functional inequalities \cite{CCL,BCL,CCH1} is well-known for the range $k\in (-d,0]$.

We classify the different behaviors in the range $k>0$ by the direct method of calculus of variations. We take advantage of the competing scalings and homogeneities in the functional $\mF_{m,k}$. The main results of this paper are summarized in Fig. \ref{Fig_mk}. If $m>1$, we show that the functional $\mF_{m,k}$ has a unique compactly supported minimizers (zone III in Fig. \ref{Fig_mk}). If $0<m<\frac{d}{d+k}$ the functional $\mF_{m,k}$ is not bounded below (zone I in Fig. \ref{Fig_mk}). The red solid line in Fig. \ref{Fig_mk} corresponds to the fair competition regime studied in \cite{CCH1,CCHCetraro}.

The main result of this paper are related to zone II in Fig. \ref{Fig_mk} corresponding to the range of parameters: $\frac{d}{d+k}<m<1$ and $k>0$. Section 3 shows that the free energy $\mF_{m,k}$ is bounded below in this parameter region. As in other similar situations as in \cite{CCH1,CCHCetraro,CCL}, the boundedness from below of the free energy $\mF_{m,k}$ implies certain functional inequalities. In our case, the inequalities obtained are the reversed HLS inequalities. In particular, we recover the inequality obtained in \cite{NN2017} as a particular curve in our range. 

Section 4 shows that, up to translations, all the minimizers of the free energy $\mF_{m,k}$ are radially decreasing probability measures with a possible condensation at the origin. This condensation is avoided in the range $m>\frac{d}{d-2}$ in zone II. The proof involves exploiting the Euler-Lagrange conditions for minimizers of $\mF_{m,k}$ in the set of probability measures and analyzing their integrability near the origin. This argument shows that minimizers have to be bounded functions in the aforementioned range. In particular, this implies that condensation can not occur in dimensions 1 and 2. Using softer arguments, we obtain another region where minimizers need to be absolutely continuous. However, for dimension $d\geq 3$ our arguments are not conclusive and condensation may happen for large values of $k$ and $m$ close to zero. The end of Section 4 contains an example of condensation in the related problem of minimizing entropy plus potential energy with a strongly confining potential. Finally, we point out that the minimizers of the free energy $\mF_{m,k}$ are equivalent to the optimizers of the reversed HLS inequality.

We note that the authors in \cite{DFH} show the reversed HLS inequality in zone II in Fig. 1 based on a relaxed variational problem with complementary techniques to ours.

\begin{figure}[ht!]
\centering
\begin{tikzpicture}
      \draw[->] (-0.2,0) -- (10,0) node[right] {$k$};
      \draw[->] (0,-0.2) -- (0,4) node[above] {$m$};
      
    \draw[scale=2.5,fill=gray!50!white] plot[smooth,samples=100,domain=0:0.43] ({\x},{1/(1+\x)}) -- 
    plot[smooth,samples=100,domain=0.43:0] (\x,{1});
      \draw[draw=gray!50!white,fill=gray!50!white] 
    plot[smooth,samples=100,domain=1:10] (\x,{1.75}) -- 
    plot[smooth,samples=100,domain=10:1] (\x,{2.5});
      
      \draw[scale=2.5,domain=0:1.2,smooth,variable=\x,red,thick] plot ({\x},{1-\x}) node[right] {$m=1-\frac{k}{d}$};
      \draw[scale=2.5,domain=0:4,smooth,variable=\x,black,thick]  plot ({\x},{1/(1+\x)}) node[right] {$m=\frac{d}{d+k}$};
      \draw[scale=2.5,domain=0:4,smooth,variable=\x,red, dashed]  plot ({\x},{3.5/(3.5+\x)}) node[right] {$m=\frac{d}{2d+k}$};
      \draw[thick](0,2.5) -- (10,2.5) node[right] {$m=1$};
      \draw[thick,blue] (0,1.75) -- (10,1.75) node[right] {$m=\frac{d-2}d$};
      
      \draw (2.4,-0.2) node {$d$};
      \draw (1.7,0.75) node {{\large \bf I}};
      \draw (6.5,1.5) node {{\large \bf II}};
      \draw (2.7,3.3) node {{\large \bf III}};
\end{tikzpicture}

\caption{Different parameter regions highlighting our main results. The red solid line corresponds to the fair-competition case. Zone I: the free energy \eqref{freeen} is not bounded below, there is a density for which the free energy is $-\infty$. Zone II: the free energy is bounded below, achieves its minimum and the reversed HLS inequality holds. Zone III: the free energy is bounded below and achieves its minimum at a compactly supported smooth function. Above the blue line minimizers of the free energy in zones II and III are absolutely continuous. The dashed red line corresponds to the reversed HLS inequality proved in \cite{NN2017}. In the shaded area of zone II the optimizers of the reversed HLS inequality are bounded and in the complement we cannot discard condensation for the optimizers except in other few cases treated in Section 4.
}
\label{Fig_mk}
\end{figure}
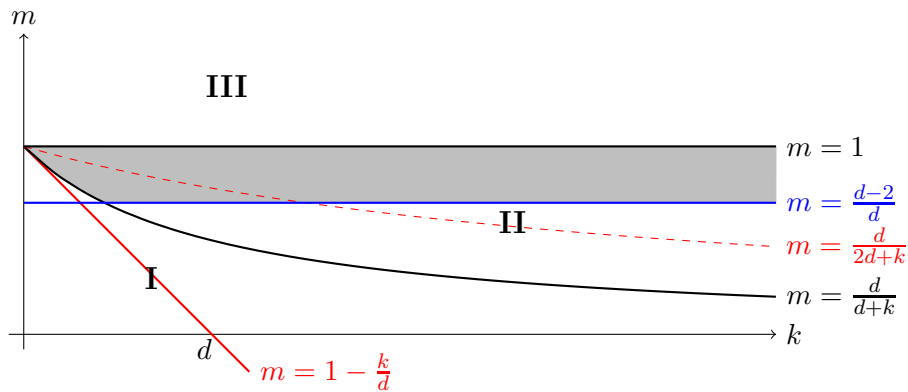

\section{Preliminaries}
We consider the space of probability measures in $\R^d$ denoted by $\prob$. Given $m\in(0,1)$ and any $\rho\in C^\infty_c(\R^d)\cap \prob$, we consider the $m$-th entropy function given by
\begin{equation*}
\mathcal{E}_m[\rho]=\ird \frac{\rho^m(x)}{m-1}\;dx.
\end{equation*}
Due to the convexity of the function $f(p)=p^m/(m-1)$, we know that \eqref{entropy} is weakly lower-semicontinuous, see \cite[Theorem 2.34]{AFP}. Exploiting weak lower-semicontinuity we can extend the functional to the whole of $\prob$ in the following way. Given $\rho\in\prob$, we define
\begin{equation*}
\Em[\rho]=\inf_{\substack{\{\rho_n\}_{n\in\N}\subset C_\mathrm{c}^\infty(\R^d)\cap\mathcal{P}(\R^d)\\ \mathrm{s.t.}\;\rho_n\rightharpoonup\rho}}\liminf_{n\to\infty}\Em(\rho_n).
\end{equation*}
It is classically known (see  \cite[Section 3.3]{buttazzo1989semicontinuity} for a complete discussion) that due to the sub-linearity of the functional the relaxation of this functional is given by the following formula. Given $\rho\in\prob$, we consider the decomposition $\rho=\rho_{ac}+\rho_s$, where $\rho_{ac}$, $\rho_s$ are the absolutely continuous and singular part with respect to the Lebesgue measure respectively. In this case, the $m$-th entropy function is given by
\begin{equation}\label{entropy}
\Em[\rho]=\ird \frac{\rho_{ac}^m(x)}{m-1}\;dx.
\end{equation}

Given $k>0$ and $\rho\in C^\infty_c(\R^d)\cap \prob$ we consider the $k$-th interaction energy given by
\begin{equation*}
	I_k[\rho]=\iint_{\RR^d\times\RR^d} |x-y|^k\rho(x)\rho(y)\;dxdy,
\end{equation*}
and $k$-th moment given by
\begin{equation*}
	J_k[\rho]:=\int_{\RR^d}|x|^k\rho(x)\;dx.
\end{equation*}
Because the function $W(x)=|x|^k$ is lower semicontinuous and positive these functionals are weakly lower-semincontinuous ( see \cite[Proposition 7.1-7.2]{santambrogio2015optimal}) and can be extended appropriately to $\prob$.

For any radially symmetric non-increasing function $\rho\in C^\infty_c(\R^d)\cap \prob$, we have the inequality
\begin{equation}\label{IJbound}
    J_k[\rho]\leq I_k[\rho] \leq 2\max\{1,2^{k-1}\}J_k[\rho]\end{equation}
which was shown in \cite[Lemma 4.1]{CCH1}. For general $\rho\in C^\infty_c(\R^d)\cap\prob$, it was proven in \cite[Lemma 2.7]{CDP} by compactness that there exists $\gamma>0$
\begin{equation}\label{gamma bound I J}
\gamma J_k[\rho]\leq I_k[\rho] \leq 2\max\{1,2^{k-1}\}J_k[\rho].
\end{equation}
In the general case, the constant $\gamma$ for the lower bound is not known explicitly. The family of free energies $\mF_{k,m}:\prob\to[-\infty,+\infty]$ can be written as:
\begin{equation*}
\mF_{m,k}[\rho]=\Em[\rho]+\frac1{2k} I_k[\rho],
\end{equation*}
where we interpret $\mF_{m,k}$ as the lower-semicontinuous extension of the restriction to $\prob\cap C_c^\infty (\R^d)$.

Now we discuss radially decreasing rearrangements. One can adapt \cite[Theorem 2.10]{burchard2009short} to our needs or use the Riesz rearrangement inequality \cite[Theorem 1.4.1]{kesavan}. 
\begin{lemma}
For all $\rho \in\prob$, we have $I_k[\rho^*]\le I_k[\rho]$, where $\rho^*$ is the radially symmetric decreasing rearrangement of $\rho$.
\end{lemma}
\begin{proof}
The result is obtained by applying the Riesz rearrangement inequality \cite[Theorem 1.4.1]{kesavan} to the functional
$$
\iint_{\RR^d\times\RR^d} ((2R)^k-|x-y|^k)_+\rho(x)\rho(y)\;dxdy,
$$ 
defined on $\rho\in\prob\cap C_c^\infty (\R^d)$ with support in $B_{R}$, the euclidean ball centered at the origin of radius $R$ and proceed by approximation arguments. 
\end{proof}

Using this lemma and the properties of radial decreasing rearrangements we deduce that
\begin{equation*}
\Em[\rho^*]=\Em[\rho]\quad\mbox{and}\quad I_k[\rho^*]\le I_k[\rho],
\end{equation*}
where $\rho^*$ is the radially decreasing rearrangement of $\rho$. This implies that the family $\mF_{k,m}$ decreases under decreasing rearrangements.

\section{Reversed HLS inequality}
First we show the following auxiliary result, which is fundamental to showing the inequality.
\begin{proposition}\label{prop:bound}
Given $\chi>0$, $0<m<1$ and $k>\frac{d(1-m)}{m}$, then there exists $C(\chi,m,k,d)\in \R$ such that
\begin{equation*}
\inf_{\rho\in\mathcal{P}(\R^d)} \Em[\rho]+\chi J_k[\rho] \ge C.
\end{equation*}
\end{proposition}
\begin{remark}
The constant $C(\chi,m,k,d)$ from Proposition~\ref{prop:bound} is explicitly calculated in its proof, see \eqref{eq for C}.
\end{remark}
We obtain Proposition~\ref{prop:bound}, by means of decomposing dyadically the respective energies, which is encoded in the following Lemma. 

\begin{lemma}\label{Lemma Dyadic Bounds}
Given $\rho\in C^\infty_c\cap\prob$, define
\begin{equation*}
\rho_j=\int_{B_{2^{j}}\setminus B_{2^{j-1}}}\rho(x)\;dx\qquad\mbox{for any $j\in \Z$}.
\end{equation*}
Then,
\begin{equation}\label{moment bound}
\frac{1}{2^k}\sum_{j\in\Z} 2^{jk}\rho_j\le J_k[\rho]\le \sum_{j\in\Z} 2^{jk}\rho_j.
\end{equation}
and
\begin{equation}\label{entropy bound}
\Em[\rho]\ge \frac{\omega_d^{1-m}}{2^{d(1-m)}} \sum_{j\in\Z} 2^{jd(1-m)}\frac{\rho_j^m}{(m-1)}
\end{equation}
\end{lemma}
\begin{proof}[Proof of Lemma~\ref{Lemma Dyadic Bounds}]
First, we show \eqref{moment bound} by decomposing the integral dyadically
\begin{equation*}
 J_k[\rho]=\ird |x|^k\rho(x)\;dx=\sum_{j\in\Z} \int_{B_{2^{j}}\setminus B_{2^{j-1}}}|x|^k\rho(x)\;dx.
\end{equation*}
Next, we notice the bounds $2^{(j-1)k}\le |x|^k\le 2^{jk}$ on $B_{2^{j}}\setminus B_{2^{j-1}}$, which imply the bounds
\begin{equation*}
2^{(j-1)k}\rho_j\le \int_{B_{2^{j}}\setminus B_{2^{j-1}}}|x|^k\rho(x)\;dx\le 2^{jk}\rho_j.
\end{equation*}
Adding these inequalities in $j$ we obtain \eqref{moment bound}.

Similarly, to show \eqref{entropy bound} we start by decomposing the integral dyadically
\begin{equation*}
\Em[\rho]=\ird \frac{\rho^m}{m-1}\;dx=\sum_{j\in\Z}\int_{B_{2^{j}}\setminus B_{2^{j-1}}}\frac{\rho^m}{m-1}\;dx.
\end{equation*}
Next, we use Jensen's inequality with the hypothesis $0<m<1$, to bound
\begin{equation*}
\begin{array}{rl}
\displaystyle\frac{|B_{2^{j}}\setminus B_{2^{j-1}}|}{|B_{2^{j}}\setminus B_{2^{j-1}}|}\int_{B_{2^{j}}\setminus B_{2^{j-1}}}\rho^m\;dx
&\displaystyle\le |B_{2^{j}}\setminus B_{2^{j-1}}|\left(\frac{1}{|B_{2^{j}}\setminus B_{2^{j-1}}|}\int_{B_{2^{j}}\setminus B_{2^{j-1}}}\rho\;dx\right)^m\\
&\displaystyle =|B_{2^{j}}\setminus B_{2^{j-1}}|^{1-m}\rho^m_j.
\end{array}
\end{equation*}
Adding these inequalities in $j$ we obtain \eqref{entropy bound}.
\end{proof}
\begin{proof}[Proof of Proposition~\ref{prop:bound}]
From \eqref{entropy bound}, we get that for any $\rho\in C^\infty_c(\R^d)\cap \prob$,
\begin{equation*}
\Em[\rho]\ge \frac{\omega_d^{1-m}}{2^{d(1-m)}(m-1)} \sum_{j\in\Z} 2^{jd(1-m)}\rho_j^m.
\end{equation*}
First, using that $\rho_j\le1$ for any $j\in \Z$, we bound
\begin{equation*}
\sum_{j\le 0} 2^{jd(1-m)}\rho_j^m\le \sum_{j\le 0} 2^{jd(1-m)}=\frac{1}{1-2^{-d(1-m)}}.
\end{equation*}
Hence, we get
\begin{equation}\label{aux bound for entropy}
\Em[\rho]\ge \frac{\omega_d^{1-m}}{2^{d(1-m)}(m-1)}\left(\frac{1}{1-2^{-d(1-m)}}+\sum_{j>0} 2^{jd(1-m)}\rho_j^m\right).
\end{equation}
We choose $a=\left(m+\frac{d(1-m)}{k}\right)/2$ to satisfy $\frac{d(1-m)}{k}<a<m$. We bound the sum by applying triple H\"older with exponents $a^{-1}$, $(m-a)^{-1}$ and $(1-m)^{-1}$, to obtain
\begin{equation*}
\sum_{j>0} 2^{jd(1-m)}\rho_j^m\le \left(\sum_{j>0} 2^{jk}\rho_j\right)^a \left(\sum_{j>0} \rho_j \right)^{m-a}\left(\sum_{j>0} 2^{j\frac{d(1-m)-ak}{1-m}}\right)^{1-m}.
\end{equation*}
Using the mass condition, the negativity of the exponent $\frac{d(1-m)-ak}{1-m}$ by our choice of $a$ and the bound of the $k$-th moment \eqref{moment bound},  we can conclude
\begin{equation*}
\sum_{j>0} 2^{jd(1-m)}\rho_j^m\le \left(\frac{1}{1-2^\frac{d(1-m)-ak}{1-m}}\right)^{1-m}2^{ka}J_k[\rho]^a.
\end{equation*}
By Young's inequality we get that for any $\delta>0$, we have the bound
\begin{equation*}
\sum_{j>0} 2^{jd(1-m)}\rho_j^m\le \frac{(1-a)a^{\frac1{1-a}}2^{k\frac{a^2}{a-1}}\delta^{-\frac{a}{1-a}} }{\left(1-2^\frac{d(1-m)-ak}{1-m}\right)^{\frac{1-m}{1-a}}} +\delta J_k[\rho].
\end{equation*}
Updating \eqref{aux bound for entropy}, we get
\begin{equation*}
\Em[\rho]\ge \frac{\omega_d^{1-m}}{2^{d(1-m)}(m-1)}\left(\frac{1}{1-2^{-d(1-m)}}+\frac{(1-a)a^{\frac1{1-a}}2^{k\frac{a^2}{a-1}}\delta^{-\frac{a}{1-a}}}{\left(1-2^\frac{d(1-m)-ak}{1-m}\right)^{\frac{1-m}{1-a}}}+\delta J_k[\rho]\right).
\end{equation*}
We pick $\delta=-\frac{2^{d(1-m)}(m-1)}{\omega_d^{1-m}}\chi>0$ and we define
\begin{equation}\label{eq for C}
C(\chi,m,k,d):=\frac{\omega_d^{1-m}}{(2^{d(1-m)}-1)(m-1)}\left(1+\frac{(1-a)a^{\frac1{1-a}}2^{k\frac{a^2}{a-1}}(1-2^{-d(1-m)})\delta^{-\frac{a}{1-a}}}{\left(1-2^\frac{d(1-m)-ak}{1-m}\right)^{\frac{1-m}{1-a}}}\right),
\end{equation}
where we remember that $a=\left(m+\frac{d(1-m)}{k}\right)/2$. Then, we get 
\begin{equation*}
\Em[\rho]+\chi J_k[\rho]\ge C(\chi,m,k,d)\,,
\end{equation*}
for any $\rho\in C^\infty_c(\R^d)\cap \prob$.
The proof for every $\rho\in\prob$ follows by the extension of the energy.
\end{proof}

\begin{theorem}\label{thm energy bounded below}
Given $0<m<1$ and $k>\frac{d(1-m)}{m}$ or equivalently $m>\frac{d}{d+k}$ with $k>0$, then for any $\rho\in \prob$
\begin{equation*}
\mF_{m,k}[\rho]\ge  \frac{C(1,m,k,d)}{2k},
\end{equation*}
where $C(1,m,k,d)$ is obtained in Proposition~\ref{prop:bound}.
\end{theorem}
\begin{proof}
Using that the energy decreases for radially decreasing rearrangements, we only need to show the inequality in the class of radially symmetric decreasing probability measures. From the bound \eqref{IJbound}, we have that for any radially symmetric decreasing probability measure
\begin{equation*}
\mF_{m,k}[\rho]\ge \Em[\rho]+\frac1{2k}J_k[\rho]\ge \frac{C(1,m,k,d)}{2k},
\end{equation*}
where the last bound comes from Proposition~\ref{prop:bound}.
\end{proof}

\begin{corollary}\label{cor:HLS}
Given $0<m<1$, $k>0$ and $m>\frac{d}{d+k}$, for any $\psi\in C^{\infty}_c(\R^d)$ positive, we have the inequality
\begin{equation}\label{HLS inequality}
\ird \ird\!\!|x-y|^k\psi(x)\psi(y)\;dxdy\ge C_0 \left(\ird |\psi(x)|\;dx \right)^{2-\frac{mk}{d(1-m)}}\!\!\left(\ird |\psi(x)|^{m}\;dx \right)^{\frac{k}{d(1-m)}},
\end{equation}
where $C_0=C_0(m,k,d)>0$, due to Theorem~\ref{thm energy bounded below}, is given by
\begin{equation*}
C_0^{\frac{d(1-m)}{k-d(1-m)}}:=\frac{\left(\frac{d(1-m)}{k}\right)^\frac{d(1-m)}{k-d(1-m)}-\left(\frac{d(1-m)}{k}\right)^\frac{k}{k-d(1-m)}}{\ds\left|\inf_{\prob}\mF_{m,k}\right|}(1-m)^{-\frac{k}{k-d(1-m)}} (2k)^{\frac{d(1-m)}{k-d(1-m)}}.
\end{equation*}
\end{corollary}

\begin{proof}
Given $r>0$, we consider the family of probability measures $\psi_r=r^{-d}\frac{\psi(x/r)}{\|\psi\|_1}$. Using Theorem~\ref{thm energy bounded below}, we have 
\begin{equation}\label{ineq-crucial}
\begin{array}{rl}
\displaystyle \inf_{\prob}\mF_{m,k}
&\displaystyle\le \inf_{r\in(0,\infty)}\mF_{m,k}[\psi_r]
= \inf_{r\in(0,\infty)}r^{d(1-m)}\Em[\psi_1]+\frac{r^k}{2k} I_k[\psi_1]\\
&\displaystyle=-C_1(m,k)|\Em[\psi_1]|^{\frac{k}{k-d(1-m)}} I_k[\psi_1]^{-\frac{d(1-m)}{k-d(1-m)}},
\end{array}
\end{equation}
where we have used that the hypothesis $k>d(1-m)/m>d(1-m)$ and
\begin{equation*}
C_1(m,k)=\left[\left(\frac{d(1-m)}{k}\right)^\frac{d(1-m)}{k-d(1-m)}-\left(\frac{d(1-m)}{k}\right)^\frac{k}{k-d(1-m)}\right] (2k)^{\frac{d(1-m)}{k-d(1-m)}}.
\end{equation*}
Therefore,
\begin{equation*}
\left|\inf_{\prob}\mF_{m,k}\right|I_k[\rho_1]^{\frac{d(1-m)}{k-d(1-m)}}\ge C_1(m,k)|\Em[\psi_1]|^{\frac{k}{k-d(1-m)}}.
\end{equation*}
Using the definition of $\Em$ and of $I_k$, we get
\begin{align*}
\left(\ird |\psi(x)|\;dx \right)^{-2}\!\!\!\ird \ird |x-y|^k\psi(x)\psi(y)&\;dxdy \ge\\
\left(\frac{C_1(m,k)}{\left|\inf_{\prob}\mF_{m,k}\right|}\right)^{\frac{k-d(1-m)}{d(1-m)}}\!\!\!&\left(\ird \frac{|\psi(x)|^{m}}{1-m}\;dx \right)^{\frac{k}{d(1-m)}}\!\!\left(\ird |\psi(x)|\;dx \right)^{-\frac{mk}{d(1-m)}}.
\end{align*}
Re-arranging terms we obtain the desired reversed HLS inequality \eqref{HLS inequality}.
\end{proof}

\begin{remark}\label{optimizers}
Optimizers of the reversed HLS inequality \eqref{HLS inequality} are equivalent to minimizers of the free energy $\mF_{m,k}$ due to the first step of the proof above \eqref{ineq-crucial}.
\end{remark}

\begin{remark}
For any $0<m<p<1$, we have by H\"older's inequality
\begin{equation*}
\ird |\psi|^p\;dx\le\left(\ird |\psi|\;dx\right)^{\frac{p-m}{1-m}} \left(\ird |\psi|^m\;dx\right)^{m\frac{1-p}{1-m}}.
\end{equation*}
Hence, by changing the homogeneities, we can replace the $L_1$ and $L_m$ integrals on the right hand side of \eqref{HLS inequality} by any $L_p$ and $L_q$ integral, with $p$, $q$ satisfying $m<p,q<1$.
\end{remark}

%

\section{Optimizers}
In this section, we show the existence of optimizers in $\prob$ and we derive conditions for the optimizers to be absolutely continuous.
\begin{theorem}\label{main theorem}
Given $0<m<1$ and $k>\frac{d(1-m)}{m}$ or equivalently $m>\frac{d}{k+d}$ with $k>0$, then for any $\chi>0$ there exists $\rho_\infty\in\prob$ with bounded $k$-th moment, such that
\begin{equation*}
\mF_{m,k}[\rho_\infty]=\inf_{\rho\in\prob}\mF_{m,k}[\rho].
\end{equation*}
They also coincide with the optimizers of the reversed HLS inequality \eqref{HLS inequality}. Moreover, there exists $a\in(0,1]$ and $\mu_\infty\in L^1(\R^d)\cap \prob$ radially symmetric and non-increasing, such that 
\begin{equation*}
\rho_\infty=(1-a)\delta_0+a\mu_\infty,
\end{equation*}
Furthermore, if $a\in(0,1)$, then the optimizer satisfies the Euler-Lagrange condition
\begin{equation}\label{EL condition}
a^{m}\frac{m}{m-1}\mu_\infty^{m-1}(x)+a^2\ird \frac{|x-y|^k}{k} \mu_\infty(y) \,dy+\frac{(1-a)a}{k}|x|^k=\frac{a^2}{k}J_k[\mu_\infty]\,,
\end{equation}
for any $x\in\R^d$. If $a=1$, then there exists $C\in\R$ such that
\begin{equation}\label{EL for a=1}
\frac{m}{m-1}\rho_\infty^{m-1}(x)+\ird \frac{|x-y|^k}{k} \rho_\infty(y) \,dy=C\,,
\end{equation}
for any $x\in\R^d$.
\end{theorem}
\begin{remark}
The free energy $\mF_{m,k}$ is geodesically convex in the 2-Wasserstein metric if $m>1-\frac{1}{d}$ and $k\ge1$. This implies the uniqueness of minimizers.
\end{remark}
\begin{proof}
We first notice that from Proposition~\ref{prop:bound} and the relationship \eqref{IJbound} between the interaction energy and potential energy there exists $C((4k)^{-1},m,k,d)$, such that for any $\rho$ radially symmetric and decreasing
\begin{equation*}
    \mF_{m,k}(\rho)\ge C((4k)^{-1},m,k,d)+\frac{1}{4k} \ird |x|^k\; d\rho(x).
\end{equation*}

Because the energy decreases under radial decreasing rearrangement, we can considering a minimizing sequence $\{\rho_n\}_{n\in\N}\subset\prob$ which is smooth and radially decreasing. Using the mass constraint and the radially decreasing property we obtain the uniform pointwise bound
\begin{equation}\label{uniform point bound}
    \rho_n(x)\le (\omega_d |x|)^{-d}.
\end{equation}
As $\{\rho_n\}_{n\in\N}$ is a minimizing sequence, we have
\begin{equation*}
  \lim_{n\to\infty}\mF_{m,k}(\rho_n)=\inf_{\rho\in\prob}\mF_{m,k}<\infty,
\end{equation*}
which implies that for every $n$ large enough $\mF_{m,k}(\rho_n)\le \inf_{\rho\in\prob}\mF_{m,k}+1$. Hence, for $n$ large enough
\begin{equation*}
    \frac{1}{4k} \ird |x|^k\; d\rho_n(x)\le \inf_{\rho\in\prob}\mF+C((4k)^{-1},m,k,d)+1.
\end{equation*}
Therefore, $\{\rho_n\}_{n\in\N}$ is tight, which implies by Prokhorov's that there exists $\rho_\infty\in\prob$ with $k$-th moment bounded and subsequence $\{\rho_{n_i}\}_{i\in\N}$, such that $\rho_{n_i}\rightharpoonup\rho_\infty$. By lower-semicontinuity of $\mF$, we get
\begin{equation*}
    \mF_{m,k}(\rho_\infty)\le \liminf_{i\to\infty}\mF_{m,k}(\rho_{n_i})=\inf_{\rho\in\prob}\mF_{m,k}.
\end{equation*}
Hence, we conclude 
\begin{equation*}
    \mF_{m,k}(\rho_\infty)=\inf_{\rho\in\prob}\mF_{m,k}.
\end{equation*}
Using that $\{\rho_n\}_{n\in\N}$ is radially decreasing and the pointwise bound \eqref{uniform point bound}, we get that there exist $a\in[0,1]$ and $\mu_\infty\in \prob\cap L^1(\R^d)$ radially symmetric and decreasing such that
\begin{equation*}
    \rho_\infty=(1-a)\delta_0+a \mu_\infty.
\end{equation*}
Using this decomposition we can re-write the energy of $\rho_\infty$ as
\begin{equation}\label{energyofrhoinfty}
  \mF_{m,k}(\rho_\infty)=a^m\Em[\mu_\infty]+\frac{a^2}{2k}I_k[\mu_\infty]+\frac{a(1-a)}{k}J_k[\mu_\infty].
\end{equation}
The case $a=0$ is discarded, because we can always find a positive radius such that the scaled indicator of the ball of that radius has negative energy, and thus it is a better competitor than the Dirac Delta.

Next, we show that $\supp(\mu_\infty)=\R^d$. Due to the optimality of $\rho_\infty$ for the decomposed free energy \eqref{energyofrhoinfty}, we get that
\begin{equation*}
\mu_\infty=\argmin_{\nu\in\prob} \left\{ a^m\Em[\nu]+\frac{a^2}{2k}I_k[\nu]+\frac{a(1-a)}{k}J_k[\nu] \right\}.
\end{equation*}
Hence, taking variations like in \cite{BCLR2,CDM,CCH1,CDP}, we get that there exists $C\in \R$, such that
\begin{equation}\label{EL1}
\begin{cases}
\displaystyle ma^m\frac{\mu_\infty^{m-1}(x)}{m-1}+\frac{(1-a)a}{k}|x|^k+a^2\ird \frac{|x-y|^k}{k} \mu_\infty(y)\,dy=C&\mbox{on $\supp(\mu_\infty)$}\\[4mm]
\displaystyle ma^m\frac{\mu_\infty^{m-1}(x)}{m-1}+\frac{(1-a)a}{k}|x|^k+a^2\ird \frac{|x-y|^k}{k} \mu_\infty(y)\,dy\ge C&\mbox{on $\R^d$}.
\end{cases}
\end{equation}
By using that $\mu_\infty$ has finite $k$-th moment and the second condition in \eqref{EL1}, it follows that $\mu_\infty$ can not vanish. If $a=1$, then \eqref{EL for a=1} follows.

If we assume that $a\in (0,1)$, we can consider the following type of perturbations.  Given $\phi\in C^\infty_c(\R^d)$ satisfying $\ird\phi\;d\mu=1$, we have that for any $\epsilon<\|\phi\|_\infty$
\begin{equation*}
	\rho_\epsilon=(1-a(1+\epsilon))\delta_0+a(1+\epsilon\phi)\mu_\infty\in\prob.
\end{equation*}
Moreover, $\rho_\epsilon\rightharpoonup\rho_\infty$ as $\epsilon\to0$. Because, $\rho_\infty$ is a minimizer we get that
\begin{equation*}
\lim_{\epsilon\to 0}\frac{\Em[\rho_\epsilon]-\Em[\rho_\infty]}{\epsilon}=0,
\end{equation*}
if the limit exists. Replacing,
\begin{align*}
\Em[\rho_\epsilon]-\Em[\rho_\infty]=\,& \,
a^m\ird((1+\epsilon\phi)^m-1) \frac{\mu_\infty^m(x)}{m-1}\,dx\\&+\frac{a}{k}\left((1-a-\epsilon a)\ird|x|^k(1+\epsilon\phi)\mu_\infty-(1-a)\ird |x|^k\mu_\infty(x)\,dx \right)\\
&+\frac{a^2}{2k}\ird\ird |x-y|^k \big( (1+\epsilon\phi(x))(1+\epsilon\phi(y))-1\big)\mu_\infty(x)\mu_\infty(y)\,dy\,dx\,.
\end{align*}
Hence, dividing by $\epsilon$ and taking the limit in the previous equation, we get the identity 
$$
\ird \!\left(ma^m\frac{\mu_\infty^{m-1}(x)}{m-1}+\frac{(1-a)a}{k}|x|^k+a^2\ird \frac{|x-y|^k}{k} \mu_\infty(y)\,dy-\frac{a^2}{k}J_k[\mu_\infty]\right)\phi(x) \mu_\infty(x)\,dx=0
$$
for any $\phi\in C^\infty_c(\R^d)$ with $\ird \phi\;d\mu_\infty=1$. Therefore, we obtain that for every $x\in\R^d$
\begin{equation*}
\frac{a^2}{k}J_k[\mu_\infty]=ma^m\frac{\mu_\infty^{m-1}(x)}{m-1}+\frac{(1-a)a}{k}|x|^k+a^2\ird \frac{|x-y|^k}{k} \mu_\infty(y)\,dy,
\end{equation*}
which is the desired \eqref{EL condition}.

Finally, let us point out that the existence of minimizers of the free energy $\mF_{m,k}$ is equivalent to the existence of optimizers of the reversed HLS inequality \eqref{HLS inequality} as discussed in Remark \ref{optimizers}.
\end{proof}

A similar proof gives the optimizers for rescaled energies.

\begin{corollary}\label{thm:minresc}
Given $m\in(d/(2+d),1)$, $k>0$, we consider 
\begin{equation*}
\mFr(\rho)=\Em[\rho]+ \frac{1}{2k}I_k[\rho] +\frac{1}{2}J_2[\rho].
\end{equation*}
Then, there exists $\rho_\infty\in\prob$ with bounded $\max(k,2)$-th moment, such that
\begin{equation*}
\mF_{m,k}[\rho_\infty]=\inf_{\rho\in\prob}\mF_{m,k}[\rho]>-\infty.
\end{equation*}
\end{corollary}

\begin{remark}
If a minimizers of $\mFr$ is regular, then using a Barenblatt rescaling we obtain self similar solutions to
\begin{equation*}
\partial_t\mu =\frac{m}{m-1}\nabla\cdot\left(\mu\left(\nabla \mu^{m-1}-\nabla\frac{|x|^k}{k}*\mu\right)\right).
\end{equation*}
For a more complete discussion see \cite{CCH1}.
\end{remark}

\subsection{Regularity of Minimizers}
Here we analyze the restriction in the parameter space that prevent that minimizers condensate.
\begin{proposition}
Under the hypothesis of Theorem~\ref{main theorem}. If $k>0$ and $a\in(0,1)$, then there exists $C(a,k,m,\mu_\infty)>0$, such that for any $x\in B_1$
\begin{equation}\label{asymptotic mu infty}
\frac{C}{ |x|^{\frac{\min\{2,k\}}{1-m}}}\le \mu_\infty(x).
\end{equation}
In particular, if $d(1-m)<2$ or equivalently $m>\frac{d-2}d$, then every optimizer is bounded. For $m>\frac{d-1}{d}$ and $k\geq 1$ the minimizer is unique and bounded.
\end{proposition}

\begin{proof}
Assuming that $a\in(0,1)$, from \eqref{EL condition} we get
\begin{equation}\label{mu infty}
\mu_\infty(x)=\left(\frac{1-m}{kma^m}\left(a^2\ird \left(|x-y|^k-|y|^k\right) \mu_\infty(y)\;dy+(1-a)a|x|^k\right)\right)^{-\frac{1}{1-m}}.
\end{equation}

If $k\ge 2$, we consider the auxiliary function $g:\R^d\to \R$, given by $g(z)=|z|^k$. By Taylor's expansion, we get that for any $x\in B_1$
\begin{align*}
|x-y|^k-|y|^k&\displaystyle=g(y-x)-g(y)\le-\nabla g(y)\cdot x+ \frac{1}{2}\sup_{z\in B_1(y)}D^2g(z)[x,x]\\
&\le -k y\cdot x |y|^{k-2}+C_k(|y|^{k-2}+1)|x|^2.
\end{align*}
Integrating, using the radial symmetry of $\mu_\infty$ and that the $k-2$ moment is bounded, we get
\begin{align*}\ird \left(|x-y|^k-|y|^k\right) \mu_\infty(y)\;dy&\displaystyle\le \ird  \left(-k y\cdot x |y|^{k-2}+C_k(|y|^{k-2}+1)|x|^2\right)\mu_\infty(y)\;dy\\
&= C_k (I_{k-2}[\mu_\infty]+1)|x|^2.
\end{align*}
Then, using \eqref{mu infty} we can conclude that there exists $C(a,m,k,\mu_\infty)>0$, such that 
\begin{equation*}
\mu_\infty(x)\ge \frac{C(a,m,k,\mu_\infty)}{(|x|^2+|x|^k)^{\frac{1}{1-m}}},
\end{equation*}
for any $x\in B_1$, and \eqref{asymptotic mu infty} for $k\ge 2$ follows.

If $1\le k<2$, we consider the auxiliary function $h:[0,1]\to \R$, given by $h(t)=|tx-y|^k-|y|^k$. Then, by Taylor's formula we get
\begin{equation}\label{taylor h}
\begin{array}{rl}
\displaystyle|x-y|^k-|y|^k&\displaystyle=h(1)-h(0)=h'(0)+ \int_0^1h''(s)\;ds.
\end{array}
\end{equation}
Calculating the derivatives of $h$, we get
\begin{equation}\label{derivative h}
h'(0)=-k(x\cdot y)|y|^{k-2},\qquad h''(s)=k\big(|x|^2|sx-y|^{k-2}-(2-k)((sx-y)\cdot x)^2|sx-y|^{k-4} \big).
\end{equation}
In particular, we can see that
\begin{equation*}
h''(s)\le k\left|s-\frac{|y|}{|x|}\right|^{k-2}|x|^k.
\end{equation*}
Hence, integrating and using that $1<k\le 2$ we can see that
\begin{equation}\label{integral h''}
\int_0^1h''(s)\;ds\;\le \frac{2k}{k-1}|x|^k.
\end{equation}
Therefore, using \eqref{taylor h}, \eqref{derivative h}, \eqref{integral h''} and the radial symmetry of $\mu_\infty$ we get
\begin{equation*}
\ird |x-y|^k-|y|^k\le -\ird k(x\cdot y)|y|^{k-2}\;d\mu_\infty+\frac{2k}{k-1}|x|^k=\frac{2k}{k-1}|x|^k. 
\end{equation*}
From \eqref{mu infty} we can conclude that there exists $C(a,m,k)>0$, such that
\begin{equation}\label{growth at zero k}
\mu_\infty(x)\ge \frac{C(a,m,k)}{|x|^{\frac{k}{m-1}}}.
\end{equation}

If $0<k\le 1$, we can use sublinearity to show
\begin{equation*}
|x-y|^k-|y|^k\le |x|^k,
\end{equation*}
see for instance \cite[Lemma 2.8]{CDP}. It follows by the previous arguments that there exists $C(a,m,k)$ such that
\begin{equation*}
\mu_\infty\ge \frac{C(a,m,k)}{|x|^k}.
\end{equation*}

Next, if $0<k\le 2$, then from the hypothesis of Theorem~\ref{main theorem} we have $k>d(1-m)/m>d(1-m)$. Hence, if $a\in(0,1)$, \eqref{growth at zero k} implies the contradiction 
\begin{equation*}
\infty=C\int_{B_1}\frac{1}{|x|^{\frac{k}{1-m}}}\;dx\le \ird \mu_\infty(x)\;dx=1.
\end{equation*}
The case $d(1-m)\le 2$ and $k\ge 2$ is analogous. Therefore, we deduce that $\rho_\infty$ is absolutely continuous and $a=1$. Moreover, if a radially decreasing optimizer is not bounded, we can use \eqref{EL for a=1} to reach a contradiction with the fact that the optimizer $\rho_\infty$ has $k$-th moment bounded by Theorem \ref{main theorem}. 
\end{proof}

\begin{proposition}
Under the hypothesis of Theorem~\ref{main theorem}. If $0<k\le1$ or $ \frac{2^{k-1}}{2^k-1} k> \frac{d(1-m)}{m}$, then $\rho_\infty$ is absolutely continuous.
\end{proposition}
\begin{proof}
We know that $\rho_\infty=(1-a)\delta_0+a\mu_\infty$ with $a\in(0,1]$. The result follows, if we can show that $a=1$ under the hypothesis.

To derive a contradiction, we assume that $a\in(0,1)$. Using that $\rho_\infty$ is the minimum of the energy, we can optimize over dilations and over the mass parameter $a$ of the optimizer $\rho_\infty$ to get
\begin{equation}\label{optimization over r}
    \partial_r\left.\left(r^{d(1-m)}\alpha^m\Em[\mu_\infty]+\frac{r^k}{k}(\alpha^2I_k[\mu_\infty]+2\alpha(1-\alpha)J_k[\mu_\infty])\right)\right|_{r=1,\alpha=a}=0,
\end{equation}
and
\begin{equation}\label{optimization a}
    \partial_\alpha\left.\left(r^{d(1-m)}\alpha^m\Em[\mu_\infty]+\frac{r^k}{k}(\alpha^2I_k[\mu_\infty]+2\alpha(1-\alpha)J_k[\mu_\infty])\right)\right|_{r=1,\alpha=a}=0.
\end{equation}
From \eqref{optimization over r} and \eqref{optimization a} multiplied by $a$, we obtain the system
\begin{equation*}
    \begin{array}{rcl}
         \displaystyle-d(1-m)a^m\Em[\mu_\infty]&\displaystyle=&\displaystyle a^2I_k[\mu_\infty]+2a(1-a)J_k[\mu_\infty], \\
         \displaystyle-mka^m\Em[\mu_\infty]&\displaystyle=&\displaystyle2a^2I_k[\mu_\infty]+2a(1-2a)J_k[\mu_\infty].
    \end{array}
\end{equation*}
Replacing we obtain
\begin{equation*}
    \frac{d(1-m)}{mk}\left(2a^2I_k[\mu_\infty]+2a(1-2a)J_k[\mu_\infty]\right)=a^2I_k[\mu_\infty]+2a(1-a)J_k[\mu_\infty].
\end{equation*}
Rearranging, we get
\begin{equation*}
    I_k[\mu_\infty] =2\left(1+\frac{\left(k-\frac{d(1-m)}{m}\right)}{a\left(2\frac{d(1-m)}{m}-k\right)}\right) J_k[\mu_\infty].
\end{equation*}
Using the inequalities \eqref{IJbound}, we deduce
\begin{equation*}
    -1/2\le\frac{\left(k-\frac{d(1-m)}{m}\right)}{a\left(2\frac{d(1-m)}{m}-k\right)}\le \max\{0,2^{k-1}-1\}.
\end{equation*}
Changing variables, we consider $\beta\in(0,1)$, such that
\begin{equation*}
    \frac{d(1-m)}{m}=\beta k,
\end{equation*}
which implies
\begin{equation*}
    -1/2\le \frac{\left(1-\beta\right)}{a\left(2\beta-1\right)}\le \max\{0,2^{k-1}-1\}.
\end{equation*}
For $k\le 1$ this condition is never satisfied, which yields the contradiction in this case. For $k>1$, using monotonicity, this condition is satisfied as long as $\beta\ge \beta_k$, where $\beta_k$ satisfies
\begin{equation*}
    \frac{1-\beta_k}{2\beta_k-1}=2^{k-1}-1,
\end{equation*}
which implies the condition
\begin{equation*}
    \gamma\ge \frac{2^{k-1}}{2^k-1},
\end{equation*}
that contradicts the hypothesis $\frac{2^{k-1}}{2^k-1}k> \frac{d(1-m)}{m}$. Hence, $a=1$ that implies the minimizers are absolutely continuous. 
\end{proof}

\subsection{Condensation for a toy model}
Here we describe the conditions on an alternative related model to show that minimizers in fact can condensate. 

Fixing $d\ge2$ and $m\in(0,1)$, we take a smooth potential $V:\R^d\to [0,\infty)$ positive radially symmetric and increasing, such that $V(0)=0$ and that
\begin{equation*}
\ird \frac{1}{V^{1-m}(x)}\;dx=M<\infty.
\end{equation*}
Then, there exists $\gamma_0(M,m)$, such that for any $\gamma<\gamma_0$ all the optimizer of
\begin{equation*}
\Em[\rho]+\gamma^{-1}\ird V(x)\;d\rho(x)
\end{equation*}
are of the form $\rho_\infty=a\delta_0+(1-a)\mu_\infty$ with $a\in(0,1)$.

We prove this by contradiction. If $a=1$, then the Euler-Lagrange conditions imply that exists $C$ such that
\begin{equation*}
\rho_\infty(x)=\frac{\gamma}{\left(\frac{m}{1-m}(V-C)\right)^{1-m}}.
\end{equation*}
Moreover, because $V$ vanishes at zero, the Euler-Lagrange conditions imply that $C\le 0$. Hence,
\begin{equation*}
\rho_\infty(x)\le \frac{\gamma}{\left(\frac{m}{1-m}V\right)^{1-m}}.
\end{equation*}
Integrating, we have
\begin{equation*}
1=\ird \rho_\infty(x)\;dx\le \ird \frac{\gamma}{\left(\frac{m}{1-m}V\right)^{1-m}}\;dx=\gamma M\left(\frac{1-m}{m}\right)^{1-m}.
\end{equation*}
Taking $\gamma$ small enough yields the contradiction. We believe this scenario may happen for the nonlinear interaction problem. However, at this moment we did not succeed in finding the right conditions for condensation of minimizers of \eqref{freeen}. This scenario also happens in related Bose-Einstein-Fokker-Planck equations, see \cite{BGT11,To12}.

\section{Zones I and III}
Next, we show that the condition $k>d(1-m)/m$ is sharp in Theorem~\ref{thm energy bounded below}. The following has already been shown in \cite{CDP}, but we include its proof here for the sake of completeness.
\begin{theorem}
Given $m<1$ and $0<k<d(1-m)/m$, then for any $\eps>0$ there exists $\rho\in\mathcal{P}(\R^d)$ such that
\begin{equation*}
\mF_{m,k}[\rho]=-\infty.
\end{equation*}
\end{theorem}
\begin{proof}
We construct a probability measure such that the entropy functional $\mathcal{E}_m$ is infinite but the interaction energy $I_k$ is bounded.
Decomposing $\R^d$ into dyadic rings, we consider
\begin{equation}\label{dyadicrho}
\rho=\sum_{j=0}^\infty \frac{\rho_j}{|B_{2^{j+1}}\setminus B_{2^{j}}|}\chi_{B_{2^{j+1}}\setminus B_{2^{j}}},
\end{equation}
where
\begin{equation*}
\rho_j=\frac{2^{-j\beta}}{\sum_{j=0}^\infty 2^{-j\beta}}.
\end{equation*}
Now we want to pick $\beta>0$ appropriately, such that $I_k[\rho]<\infty$ and $\Em[\rho]=-\infty$. By \eqref{gamma bound I J}, we have $I_k[\rho]\le J_k[\rho]$.
Using \eqref{dyadicrho}, the exact form for $\rho$, we get
\begin{equation*}
\ird |x|^k\;d\rho=\sum_{j=0}^\infty \frac{\rho_j}{|B_{2^{j+1}}\setminus B_{2^{j}}|}\int_{B_{2^{j+1}}\setminus B_{2^{j}}}\!\!\!\!|x|^k\;dx= C(d,k)\sum_{j=0}^\infty \rho_j 2^{jk}=C_1(d,k) \frac{\sum_{j=0}^\infty 2^{-j(\beta-k)}}{\sum_{j=0}^\infty 2^{-j\beta}},
\end{equation*}
where $C_1(d,k)$ is a constant that depends on the dimension and $k$. In order for the right hand side to be finite, which in turn bounds the interaction energy, we need 
\begin{equation}\label{cond1}
\beta>k.
\end{equation}
Turning our attention to the entropy and using the exact form of $\rho$ again, we get
\begin{equation*}
\ird \rho^m(x)\;dx=\sum_{j=0}^\infty \rho_j^m|B_{2^{j+1}}\setminus B_{2^{j}}|^{(1-m)}=C_2(d,k)\frac{\sum_{j=0}^\infty 2^{-j(m\beta-d(1-m))}}{\left(\sum_{j=0}^\infty 2^{-j\beta}\right)^m}.
\end{equation*}
For the right hand side to be infinite, we need
\begin{equation}\label{cond2}
m\beta<d(1-m).
\end{equation}
Therefore, combining \eqref{cond1} and \eqref{cond2}, we get
\begin{equation*}
k<\beta<\frac{d(1-m)}{m}.
\end{equation*}
By hypothesis, we have that $\beta<d(1-m)/m$, which implies we can take the midpoint and the result follows.
\end{proof}

Now we show that the case $m>1$ is much simpler than $m<1$.
\begin{theorem}[Existence]\label{thm:min}
Given $m>1$ and $k>0$, then the energy
\begin{equation*}
\mF(\rho)=\frac{1}{m-1}\ird \rho^m(x)\;dx+\frac{1}{2k} \ird\ird |x-y|^k\; d\rho(x)d\rho(y)
\end{equation*}
is bounded below and admits a minimizer $\rho_\infty\in\prob$. Furthermore, $\rho_\infty$ is radially symmetric and decreasing. Also, it is solution to the Euler-Lagrange equation: there exists $C\in\R$ and $R>0$, such that $\supp(\rho_\infty)=B_R$ and 
\begin{equation*}
    \begin{cases}
\displaystyle\frac{m}{m-1}\rho_\infty^{m-1}(x)+\frac{1}{k}\ird |x-y|^k\; d\rho_\infty(y)=C&\mbox{a.e. $x\in B_R$},\\[4mm]
\displaystyle\frac{m}{m-1}\rho_\infty^{m-1}(x)+\frac{1}{k}\ird |x-y|^k\; d\rho_\infty(y)\ge C&\mbox{a.e. $x\in \R^d$}.
    \end{cases}
\end{equation*}
\end{theorem}

\begin{proof}
The energy is reduced by symmetric decreasing rearrangements, therefore we can take a minimizing sequence $\{\rho_n\}_{n\in\N}$ which is symmetric and decreasing. Using that $m>1$, we get
\begin{align*}
\mF_{m,k}(\rho_n)&=\frac{1}{m-1}\ird \rho_n^m(x)\;dx+ \frac{1}{2k}\ird\ird |x-y|^k\; d\rho_n(x)d\rho_n(y)\\
    &\ge \frac{1}{2k}\ird\ird |x-y|^k\; d\rho_n(x)d\rho_n(y)\ge \frac{1}{2k}\ird |x|^k\; d\rho_n(x),
\end{align*}
where the last inequality follows from \eqref{IJbound}.
Therefore, for $n$ large enough we get
\begin{equation*}
   \frac{1}{2k} \ird |x|^k\; d\rho_n(x)\le \mF_{m,k}(\rho_n)\le \inf \mF_{m,k}+1.
\end{equation*}
This implies that $\rho_n$ is tight. Therefore, there exists $\rho_\infty\in \prob$ with the $k$-th moment bounded, such that $\rho_n\rightharpoonup\rho_\infty$. It follows that $\rho_\infty$ is radially symmetric and decreasing and that it solves the Euler-Lagrange equation.
\end{proof}

\begin{remark}
Theorem~\ref{thm:min} can easily generalized to any interaction potential that grows at infinity, see \cite[Lemma 2.9]{CDP}.
\end{remark}

\begin{remark}
Arguing as in Corollary~\ref{cor:HLS}, Theorem~\ref{thm:min} implies that for any positive $\psi\in C^\infty_c(\R^d)$, $m>1$ and $k>0$ we have the family of functional inequalities
\begin{equation*}
0<\inf_{\prob}\mF_{m,k}\le C(m,k) I_k[\psi]^{\frac{d(m-1)}{k+d(m-1)}}\|\psi\|_{L^1(\R^d)}^{-\frac{mk+2d(m-1)}{k+d(m-1)}}\|\psi\|_{{L^m(\R^d)}}^{\frac{mk}{k+d(m-1)}}.
\end{equation*}
\end{remark}


\bibliographystyle{abbrv}
\bibliography{Bibliography}

\end{document}